\documentclass[oneside,article]{memoir}
\usepackage{amsmath}         
\usepackage{amsthm,thmtools,thm-restate} 
\usepackage{enumitem}        

\usepackage{rotating}

\usepackage[final]{listings}
\usepackage{bbm}
\usepackage{xcolor}

\definecolor{codegreen}{rgb}{0,0.6,0}
\definecolor{codegray}{rgb}{0.5,0.5,0.5}
\definecolor{codepurple}{rgb}{0.58,0,0.82}
\definecolor{backcolour}{rgb}{0.95,0.95,0.92}

\lstdefinestyle{mystyle}{
    backgroundcolor=\color{backcolour},   
    commentstyle=\color{codegreen},
    keywordstyle=\color{magenta},
    numberstyle=\tiny\color{codegray},
    stringstyle=\color{codepurple},
    basicstyle=\ttfamily\footnotesize,
    breakatwhitespace=false,         
    breaklines=true,                 
    captionpos=b,                    
    keepspaces=true,                 
    numbers=left,                    
    numbersep=5pt,                  
    showspaces=false,                
    showstringspaces=false,
    showtabs=false,                  
    tabsize=2
}

\usepackage{tikz}
  \usetikzlibrary{matrix,arrows,positioning}
\usepackage{amsfonts}
\usepackage{parskip}
\usepackage{mathtools}

\usepackage{tikz-cd}

\usepackage{float} 
\usepackage{caption, subcaption} 

\usepackage[nobysame,alphabetic,initials]{amsrefs} 
\usepackage[hidelinks]{hyperref}                   
\usepackage[capitalize,nosort]{cleveref}           

\usepackage{amssymb}
\usepackage[mathscr]{euscript}
\usepackage{relsize}
\usepackage{stmaryrd}
\usepackage{stackengine}

\usepackage{indentfirst} 
\usepackage{multicol}

\usepackage[inline,nomargin]{fixme} 
  \fxsetup{targetlayout=color}

\usepackage{float} 

\isopage[12]

\setlrmargins{*}{*}{1}
\checkandfixthelayout

\counterwithout{section}{chapter}
\usepackage{appendix}

  \newcommand{\iso}{\mathrel{\cong}}

  \newcommand{\op}{{\mathord\mathrm{op}}}

  \newcommand{\id}[1][]{\operatorname{id}_{#1}}
  
  \renewcommand{\tilde}{\widetilde}

  \renewcommand{\hat}{\widehat}

  \newcommand{\cat}[1]{\mathscr{#1}}

  \newcommand{\from}{\colon}

  \declaretheorem[style=definition,within=section]{definition}
  \declaretheorem[style=definition,numberlike=definition]{example}
  \declaretheorem[style=definition,numberlike=definition]{remark}
  
  \declaretheorem[style=definition,numberlike=definition]{assumption}

  \declaretheorem[style=plain,numberlike=definition]{corollary}
  \declaretheorem[style=plain,numberlike=definition]{lemma}
  \declaretheorem[style=plain,numberlike=definition]{proposition}
  \declaretheorem[style=plain,numberlike=definition]{theorem}

  \declaretheorem[style=plain,numbered=no,name=Theorem]{theorem*}

  \Crefname{corollary}{Corollary}{Corollaries}
  \Crefname{definition}{Definition}{Definitions}
  \Crefname{lemma}{Lemma}{Lemmas}
  \Crefname{proposition}{Proposition}{Propositions}
  \Crefname{remark}{Remark}{Remarks}
  \Crefname{theorem}{Theorem}{Theorems}
  \Crefname{notation}{Notation}{Notations}

  \newlist{axioms}{enumerate}{1}
  \Crefname{axiomsi}{}{}

  \newenvironment{tikzeq*}
  {
    \begingroup
    \begin{equation*}
    \begin{tikzpicture}[baseline=(current bounding box.center)]
  }
  {
    \end{tikzpicture}
    \end{equation*}
    \endgroup
    \ignorespacesafterend
  }

  \tikzset
  {
    diagram/.style=
    {
      matrix of math nodes,
      column sep={4.3em,between origins},
      row sep={4em,between origins},
      text height=1.5ex,
      text depth=.25ex
    },
    over/.style={preaction={draw=white,-,line width=6pt}},
    every to/.style={font=\footnotesize},
    inj/.style={right hook->},
    surj/.style={-{Latex[open]}},
    cof/.style={>->},
    fib/.style={->>},
  }

  \DeclareFontFamily{U}{mathx}{\hyphenchar\font45}

  \DeclareFontShape{U}{mathx}{m}{n}{
    <5> <6> <7> <8> <9> <10>
    <10.95> <12> <14.4> <17.28> <20.74> <24.88>
    mathx10}{}

  \DeclareSymbolFont{mathx}{U}{mathx}{m}{n}

  \DeclareFontFamily{U}{mathb}{\hyphenchar\font45}

  \DeclareFontShape{U}{mathb}{m}{n}{
    <5> <6> <7> <8> <9> <10>
    <10.95> <12> <14.4> <17.28> <20.74> <24.88>
    mathb10}{}

  \DeclareSymbolFont{mathb}{U}{mathb}{m}{n}

  \DeclareMathAccent{\widebar}{0}{mathx}{"73}

  \DeclareMathSymbol{\Rsh}{\mathrel}{mathb}{"E9}

  \DeclareFontFamily{U}{MnSymbolA}{}

  \DeclareFontShape{U}{MnSymbolA}{m}{n}{
    <-6> MnSymbolA5
    <6-7> MnSymbolA6
    <7-8> MnSymbolA7
    <8-9> MnSymbolA8
    <9-10> MnSymbolA9
    <10-12> MnSymbolA10
    <12-> MnSymbolA12}{}

  \DeclareSymbolFont{MnSyA}{U}{MnSymbolA}{m}{n}

  \DeclareMathSymbol{\twoheaddownarrow}{\mathrel}{MnSyA}{27}

  \newcommand{\MSC}[1]{%
    \let\thempfn\relax
    \footnotetext[0]{2020 Mathematics Subject Classification: #1.}
  }

\tikzstyle{vertex}=[circle, draw, minimum size=7pt, inner sep=0pt]

\newcommand{\Ab}{\textsf{Ab}}
\newcommand{\Graph}{\mathsf{Graph}} 
\newcommand{\Set}{\mathsf{Set}} 

\newcommand{\lan}{\textsf{lan}}
\newcommand{\Colim}{\textsf{Colim}}

\newcommand{\bbG}{\mathbb{G}} 

\DeclareFontFamily{U}{dmjhira}{}
\DeclareFontShape{U}{dmjhira}{m}{n}{ <-> dmjhira }{}

\DeclareRobustCommand{\yo}{\text{\usefont{U}{dmjhira}{m}{n}\symbol{"48}}}

\author{Krzysztof Kapulkin \and Nathan Kershaw} 
\title{Closed symmetric monoidal structures on the category of graphs}
\date{September 30, 2023}

\begin{document}

  \maketitle

\begin{abstract}
We show that the category of (reflexive) graphs and graph maps carries exactly two closed symmetric monoidal products: the box product and the categorical product.
\end{abstract}



\section*{Introduction}

Several notions of products are considered in graph theory, including the box product (also occasionally called cartesian product), the Kronecker product, the lexicographic product, and the strong product, among others \cite{imrich-klavzar}.
Understanding different combinatorial properties of these products remains an active area of research within combinatorics.

In this paper, we approach this question from the categorical perspective, namely, we would like to understand which graph products define a \emph{closed symmetric monoidal} structure on the category of graphs.
Monoidal structures are a natural framework for capturing abstract products that can be considered on a given category.

Our interest in the subject comes from discrete homotopy theory, an area of mathematics concerned with applying techniques from topology, and more precisely homotopy theory, to study combinatorial properties of graphs.
Numerous such theories are now under active development, including the A-homotopy theory \cite{babson-barcelo-longueville-laubenbacher,barcelo-laubenbacher,carranza-kapulkin:cubical-graphs} and the $\times$-homotopy theory \cite{dochtermann,chih-scull}, among others.
Different models of discrete homotopy theory use different graph products, but end up proving similar results often with similar looking proofs.
This leads to the question of whether this development can be done synthetically by axiomatizing the desired properties of the product.
This paper is therefore a first step towards such synthetic theory.

Our main theorem is:

\begin{theorem*}[\cref{main th}]
  The category of (reflexive) graphs carries precisely two closed symmetric monoidal structures: the box product and the categorical product.
\end{theorem*}

Throughout the paper, we assume familiarity with category theory at the level of an introductory text, e.g., \cite{riehl:context} or \cite{maclane:cwm}.
We recall all relevant graph-theoretic notions.

This paper is organized as follows.
In \cref{sec:monoidal}, we review the background on monoidal categories and the Day convolution product, which is a technique of upgrading a promonoidal structure on a small category to its presheaf category via a left Kan extension (cf.~\cite{day,day:reflection}).
In \cref{sec:graphs}, we introduce the category of graphs and graph maps.
Our proof starts with a preliminary analysis of left Kan extensions on the relevant category of graphs in \cref{sec:Kan ext}.
Finally, in \cref{sec:main-thm}, we put all the pieces together and prove the classification of closed symmetric monoidal structures on the category of graphs.

\textbf{Acknowledgement.}
We would like to thank Daniel Carranza for many helpful conversations and insights regarding this project.

This material is based upon work supported by the National Science Foundation under Grant No.~DMS-1928930 while the first author participated in a program hosted by the Mathematical Sciences Research Institute in Berkeley, California, during the 2022--23 academic year.

The work of the second author was supported through an NSERC Undergraduate Student Research Award in the Summer of 2022.
\newcommand{\singl}{\mathbf{1}}
\section{Monoidal categories} \label{sec:monoidal}
As mentioned above, our goal is to analyze, through the lens of category theory, different products of graphs.
Monoidal structures are a natural way of doing so.
In brief, a monoidal structure on a category $\cat{C}$ is a bifunctor, usually denoted with the symbol $\otimes$, that is unital and associative (up to a natural isomorphism).
Such products can then satisfy additional properties, like symmetry or closure.

In this section, we begin by reviewing a basic theory of monoidal structures and develop some preliminary results allowing us to classify them on categories of interest.

\begin{definition} \label{def:monoidal category}
A \emph{monoidal category} consists of a category $\cat{C}$ together with:
\begin{itemize}
    \item a functor $\otimes \colon  \cat{C} \times \cat{C} \to \cat{C}$, referred to as the \emph{tensor product};
    \item an object $I \in \cat{C}$, called the \emph{unit};
    \item three natural isomorphisms $\alpha$, $\lambda$, and $\rho$ with components:
    \begin{itemize}
        \item $\alpha_{a,b,c} \colon  a \otimes (b \otimes c) \cong (a \otimes b) \otimes c$, for all $a$, $b$, $c$,
        \item $\lambda_a \colon  I \otimes a \cong a$ for all $a$,
        \item $\rho_a \colon  a \otimes I \cong a$, for all $a$,
    \end{itemize}
\end{itemize}
subject to the axioms:
\begin{enumerate}
    \item the diagram
\begin{center}
\begin{tikzpicture}
\node (P0) at (90:2.8cm) {$a\otimes \bigl( b\otimes (c\otimes d) \bigr) $};
\node (P1) at (90+72:2.5cm) {$a\otimes \bigl((b\otimes c)\otimes d\bigr)$} ;
\node (P2) at (90+2*72:2.5cm) {$\mathllap{\bigl(a\otimes (b\otimes c)\bigr)}\otimes d$};
\node (P3) at (90+3*72:2.5cm) {$\bigl((a\otimes b)\mathrlap{\otimes c\bigr)\otimes d}$};
\node (P4) at (90+4*72:2.5cm) {$(a\otimes b)\otimes (c\otimes d)$};
\draw
(P0) edge[->,>=angle 90] node[left] {$\id \otimes \alpha$} (P1)
(P1) edge[->,>=angle 90] node[left] {$\alpha$} (P2)
(P2) edge[->,>=angle 90] node[above] {$\alpha \otimes \id$} (P3)
(P4) edge[->,>=angle 90] node[right] {$\alpha$} (P3)
(P0) edge[->,>=angle 90] node[right] {$\alpha$} (P4);
\end{tikzpicture}
\end{center}
commutes for all objects $a,b,c,d \in \cat{C}$.
\item the diagram
\begin{center}
\begin{tikzcd}
a \otimes (1 \otimes b) \arrow[rr, "{\alpha_{a,1,b}}"] \arrow[rd, "\rho_x \otimes 1_y"'] &             & (a \otimes 1) \otimes b \arrow[ld, "1_a \otimes \lambda_y"] \\
 & a \otimes b &               
\end{tikzcd}
\end{center}
commutes for all objects $a, b \in \cat{C}$.
\end{enumerate}
\end{definition}

Examples of monoidal categories include any category with finite products (or finite coproducts).
Note that such a product need not be commutative. 
For instance, for a given category $\cat{C}$, the category of endofunctors on $\cat{C}$ is a monoidal category with the tensor product given by composition of functors.
    
If we wish to further restrict our study to only commutative products, we can impose extra conditions, which motivates the following definition.

\begin{definition} \label{def:symmetric monoidal}
A \emph{symmetric monoidal category} is a monoidal category $(\cat{C}, \otimes , I)$ such that for every pair of objects $a,b$, there is a natural isomorphism $s_{a,b} \colon  a\otimes b \to b \otimes a$, such that $s_{a,b} \circ s_{b,a} = 1_{b \otimes a}$, and the diagram
\begin{center}
\begin{tikzcd}
(a \otimes b) \otimes c \arrow[r, "{\alpha_{a,b,c}}"] \arrow[d, "{s_{a,b\otimes 1_c}}"] & a \otimes (b \otimes c) \arrow[r, "{s_{a,b \otimes c}}"]   & (b \otimes c) \otimes a \arrow[d, "{\alpha_{b,c,a}}"] \\
(b \otimes a) \otimes c \arrow[r, "{\alpha_{b,x,c}}"]                                   & b \otimes (a \otimes c) \arrow[r, "{1_b \otimes s_{a,c}}"] & b \otimes (c \otimes a)                              
\end{tikzcd}
\end{center}
commutes for all objects $a$, $b$, and $c$. 
\end{definition}

\begin{example}
    The category $\textsf{Ab}$ of abelian groups, along with the categorical product, is symmetric monoidal with the trivial group being the unit. 
    We can also equip $\Ab$ with the tensor product to form a symmetric monoidal structure. 
    In this case, the unit is $\mathbb{Z}$.
\end{example}

For a given monoidal category $(\cat{C}, \otimes, I)$, and object $X \in \cat{C}$, one might ask whether the functor $X \otimes - \colon \cat{C} \to \cat{C}$ admits a right adjoint.
Such a right adjoint would then act as the object of morphisms, e.g., in the case of the category of sets with the cartesian product, the right adjoint gives exactly the set of functions from $X$ to another object.
This leads to the final property of monoidal structures we shall discuss.

\begin{definition} \label{def:closed monoidal}
A symmetric monoidal category is \emph{closed} if for every object $a$ in $\cat{C}$, the functor $a \otimes - \from \cat{C} \to \cat{C}$ has a right adjoint $\textsf{hom}(a,-) \from \cat{C}\to \cat{C}$.
\end{definition}

\begin{example} \label{closed symmetric} 
The following are closed symmetric monoidal categories:
\begin{itemize}
    \item The category $\textsf{Set}$ of sets, equipped with the cartesian product. The unit is the singleton set, and given a set $X$, the right adjoint is given by $\hom(X,Y) = Y^X$.

    \item The category $\Ab$, of abelian groups, equipped with the tensor product.
    For a given group $G$, the right adjoint is given by the set of group homomorphisms $\hom(G,H)$ with the group structure defined pointwise.

    \item The categorical product on $\Ab$ is not closed.
    If it were closed, then since there is an infinite number of group homomorphisms $f \colon \mathbb{Z} \times \mathbbm{1} \to \mathbb{Z}$, we should expect an infinite number of homomorphisms $f \colon \mathbbm{1} \to \mathbb{Z}^\mathbb{Z}$. Since this is clearly not the case, the categorical product cannot be closed.

    \item We will also see in \cref{sec:graphs}, that two products of graphs, the box and categorical products, are closed symmetric monoidal (see \cref{def:graph products}).

\end{itemize}
\end{example}

We now turn our attention toward classifying monoidal structures on categories of interest.
We begin by establishing restrictions on what the unit of such a monoidal structure might be.
To do so, we work in the generality of well-pointed categories whose definition we recall.

\begin{definition} \label{def:well pointed}
    A category $\cat{C}$ with a terminal object $1$ is \textit{well pointed} if for every pair of morphisms $f,g \colon A \to B$ such that $f \neq g$, there exists a morphism $p \colon 1 \to A$ such that $f \circ p \neq g \circ p$.
\end{definition}

In other words, equality of morphisms in a well-pointed category is detected on the global sections of the domain.
Examples of well-pointed categories include the category of sets and that of topological spaces, but not the category of groups.

\begin{proposition} \label{prop:monoidal_unit}
  Let $\cat{C}$ be a well-pointed category with a monoidal structure $(\cat{C}, \otimes, I)$, and let $1$ be the terminal object of $\cat{C}$.
  Then the canonical map $I \to 1$ is a monomorphism.
\end{proposition}

\begin{proof}
  Since $\cat{C}$ is well-pointed, it suffices to show that for any pair of maps $x, y \from 1 \to I$, we have $x=y$.
  Consider the following square
      \[ \begin{tikzcd}[column sep = large]
        1 \otimes 1 
            \arrow[r,"x \otimes \id"] 
            \arrow[d,swap,"\id \otimes y"] & 
        I \otimes 1
            \arrow[d,"\id \otimes y"] \\
        1 \otimes I
            \arrow[r,"x \otimes \id"] & 
        I \otimes I
    \end{tikzcd} \]
  which commutes by functoriality of the tensor product.
  Using the fact that $I$ is the unit of the monoidal product, we observe that the square above simplifies to
    \[ \begin{tikzcd}[column sep = large]
        1 \otimes 1 
            \arrow[r,"!"] 
            \arrow[d,swap,"!"] & 
        1
            \arrow[d,"y"] \\
        1
            \arrow[r,"x"] & 
        I
    \end{tikzcd} \]
  but since there is a map $1 \iso 1 \otimes I \to 1 \otimes 1$, this square commutes exactly when $x=y$.
\end{proof}

As described above, our goal is to classify closed symmetric monoidal structures on the category of graphs, which is a reflective subcategory of a presheaf category, i.e., a category of the form $\Set^{C^{\op}}$.
For notational convenience, we will write $\hat{C}$ for $\Set^{C^{\op}}$ throughout.
Closed symmetric monoidal structures on presheaf categories are known to arise via the \emph{Day convolution product}.
The Day convolution of a promonoidal functor $F \from C \times C \to \hat{C}$ is obtained by taking the left Kan extension of $F$ along $\yo \times \yo$: 
\begin{center}
\begin{tikzcd}
C \times C \arrow[d, "\yo \times \yo"'] \arrow[rr, "F"]                          &  & \hat{C} \\
\hat{C} \times \hat{C} \arrow[rru, "\lan_{\yo \times \yo }F"', dotted] &  &             
\end{tikzcd}
\end{center}
where we write $\yo \from C \to \hat{C}$ for the Yoneda embedding.
This leads to the following theorem:
\begin{theorem}[\cite{day}]\label{TH:Day convolution}
    Let $C$ be a small category. Then there exists a bijection between the set of closed symmetric monoidal structures on $\hat{C}$ and promonoidal functors $F \colon C \times C \to \hat{C}$.
\end{theorem}

It will not be important to us what a promonoidal structure is, but a key consequence of this theorem is that to characterize closed symmetric monoidal structures on a category $\hat{C}$ we only need to consider functors $F \colon C \times C \to \hat{C}$.
To present a sample application, we show (a well known fact) that the category of sets carries a unique closed symmetric monoidal structure.

\begin{proposition} \label{ex: monoidal in set}
    The only closed symmetric monoidal structure in the category $\Set$ is the cartesian product.
\end{proposition}

\begin{proof}
    We know from \cref{closed symmetric} that the cartesian product defines a closed symmetric monoidal structure on $\Set$. 
    It remains to verify that it is the only such structure.
    First, we use \cref{prop:monoidal_unit} to show that the singleton set, call it $1$, is the unit. 
    We know in $\Set$ that $1$ is the terminal object, with subobjects $1$ and $\varnothing$. 
    Now, suppose $(\Set, \otimes, \varnothing)$ is a closed symmetric monoidal structure.
    Since $\varnothing$ is the initial object, it must be preserved by $-\otimes X$ for any set $X$. 
    But since $\varnothing$ is also the unit, we would have 
    $\varnothing \iso \varnothing \otimes X \iso X$, a contradiction. 

   Now note that $\Set \iso \Set^{[0]^\op}$, where $[0]$ is the terminal category (i.e., the category with one object and only the identity morphism). 
    We thus must consider functors $F \from [0] \times [0] \to \Set$ whose left Kan extension along $\yo \times \yo$
is a closed symmetric monoidal structure.

Writing $0$ for the unique object in $[0]$, we know that $(\yo \times \yo)(0,0) = (1,1)$. 
Using closure of the monoidal structure, $\lan_{\yo \times \yo}(F)(1,1) = 1$.
By commutativity, we must have that $F(0,0) = 1$, which defines $F$. 
Using the pointwise formula for Kan extensions, we obtain that $\lan_{\yo \times \yo}(F)(X,Y)$ is indeed the cartesian product of $X$ and $Y$.
\end{proof}

Unfortunately, the category of graphs is not a presheaf category, but rather a reflective category thereof.
For that reason, we need to strengthen \cref{TH:Day convolution} to the setting of presentable categories.
Similar strengthening was considered by Day \cite{day:reflection,day:monoidal}.

For the remainder of this section, we fix a presentable category $\cat{C}$, i.e., a category along with a full embedding $i \colon  \cat{C} \hookrightarrow \hat{C}$ that admits a left adjoint, called the \emph{reflector}, $L \colon  \hat{C} \to \cat{C}$.
In particular, $Li \iso \id$. 

We further require that $\yo$ factors as
\begin{center}
\begin{tikzcd}
C \arrow[rd, "\hat{\yo}"'] \arrow[rr, "\yo"] &                          & \hat{C} \\
  & \cat{C} \arrow[ru, "i"',hook] &        
\end{tikzcd}        
\end{center}

Now, before stating the resulting proposition from this setup, we first must state the following lemma:

\begin{lemma} \label{lemma:comma_cat_isos} 
   Let $\cat{C}$ be a reflective category as above, and fix objects $X$ and $Y$ in $\cat{C}$. Then there are isomorphisms:
    \begin{equation}
    \tag{1} \hat{\yo} \downarrow X \times \hat{\yo} \downarrow Y  \iso \hat{\yo} \times \hat{\yo} \downarrow (X,Y) 
    \end{equation}
    Defined by sending objects $\Bigl(\bigl(a \in C, f \colon \hat{\yo}(a) \to X\bigr), \bigl(b \in C, g \colon \hat{\yo}(b) \to Y\bigr)\Bigr)$ in $\hat{\yo} \downarrow X \times \hat{\yo} \downarrow Y$ to $\Bigl((a,b), (f,g) \colon \bigl(\hat{\yo}(a), 
    \hat{\yo}(b)\bigr) \to (X,Y)\Bigr)$ in $\hat{\yo} \times \hat{\yo} \downarrow (X,Y)$.

    \begin{equation}
    \tag{2} \hat{\yo} \times \hat{\yo} \downarrow (X,Y) \iso \yo \times \yo \downarrow (X,Y) 
    \end{equation}
    Defined by sending $\Bigl((a,b), (f,g) \colon \bigl(\hat{\yo}(a), \hat{\yo}(b)\bigr) \to (X,Y)\Bigr)$  to $\Bigl((a,b), (f,g) \colon \bigl(\yo(a), \yo(b)\bigr) \to (X,Y)\Bigr)$.

    \begin{equation}
    \tag{3} \yo \times \yo \downarrow (X,Y) \iso \yo \downarrow X \times \yo \downarrow Y 
    \end{equation}
    Defined by sending $\Bigl((a,b), (f,g) \colon \bigl(\yo(a), \yo(b)\bigr) \to (X,Y)\Bigr)$ to $\Bigl(\bigl(a,f \colon \yo(a) \to X\bigr),(b,g \colon \yo(b) \to Y)\Bigr)$.
    
    Note that in (2) and (3) we treat $X$ and $Y$ as elements of $\hat{C}$, and omit the inclusion $i$ in the equations. \qed
\end{lemma}

\begin{proposition}\label{Prop: monoidal as kan extensions}
    Let $C$, $\cat{C}$, and $\hat{C}$ be as above.
    Then, if $\otimes \colon  \cat{C} \times \cat{C} \to \cat{C}$ defines a closed symmetric monoidal structure on $\cat{C}$, $\otimes$ is equal to the left Kan extension of $\otimes \circ(\hat{\yo} \times \hat{\yo})$ along  $\hat{\yo} \times \hat{\yo}$:
\begin{center}
\begin{tikzcd}
C \times C \arrow[d, "\hat{\yo} \times \hat{\yo}"'] \arrow[rr, "\otimes \circ(\hat{\yo} \times \hat{\yo})"] &  & \cat{C} \\
\cat{C} \times \cat{C} \arrow[rru, dotted]                                                                  &  &        
\end{tikzcd}
\end{center}
    i.e. $\otimes \iso \lan_{\hat{\yo} \times \hat{\yo}}\bigl(\otimes \circ (\hat{\yo} \times \hat{\yo})\bigr)$.
\end{proposition}
\begin{proof}
For simplicity, we write $\tilde{\otimes}$ for $\lan_{\hat{\yo} \times \hat{\yo}}\bigl(\otimes \circ (\hat{\yo} \times \hat{\yo})\bigr)$.
Our goal is to show that $\otimes \iso \tilde{\otimes}$. 
To do this, we define another product $\overline{\otimes} \colon \hat{C} \times \hat{C} \to \cat{C}$ by $\overline{\otimes} = \lan_{\yo \times \yo}(\otimes \circ \bigl(\hat{\yo} \times \hat{\yo})\bigr)$. 
The relevant functors are displayed in the following diagram:
\begin{center}
\begin{tikzcd}
C \times C \arrow[d, "\hat{\yo} \times \hat{\yo}"'] \arrow[dd, "\yo \times \yo"', bend right, shift right=8] \arrow[rr, "\otimes \circ (\hat{\yo} \times \hat{\yo})"] &  & \cat{C} \arrow[r, "i", shift left] & \hat{C} \times \hat{C} \arrow[l, "L", shift left] \\
\cat{C} \times \cat{C} \arrow[d, "i \times i"'] \arrow[rru, "\tilde{\otimes}", dotted]                                                                                &  &                                    &                                                   \\
\hat{C} \times \hat{C} \arrow[rruu, "\overline{\otimes}"', dotted] \arrow[d, "L \times L"']                                                                           &  &                                    &                                                   \\
\cat{C} \times \cat{C} \arrow[rruuu, "\otimes"', shift right=2]                                                                                                       &  &                                    &                                                  
\end{tikzcd}
\end{center}
Our proof is divided into two parts: 
\begin{equation*}\label{eq:tilde-bar}
    \tag{i} \tilde{\otimes} \iso \overline{\otimes} \circ (i \times i) 
\end{equation*}
\begin{equation*} \label{eq:bar-nothing}
    \tag{ii} \overline{\otimes} \iso \otimes \circ (L \times L)
\end{equation*}
Once these are established, we get that: 
\begin{align*}
    \otimes &\iso \otimes \circ (L \times L) \circ (i \times i)\\
    & \iso \overline{\otimes} \circ (i \times i)\\
    & \iso \tilde{\otimes}\\
\end{align*}
as desired.

For part (\ref{eq:tilde-bar}), we fix $X, Y \in \cat{C}$.
We have that by the pointwise formula:
$$X\tilde{\otimes}Y = \Colim\bigl(\hat{\yo} \times \hat{\yo} \downarrow (X,Y) \xrightarrow{\pi} C \times C \xrightarrow{\otimes \circ (\hat{\yo} \times \hat{\yo})} \cat{C}\bigr)$$
$$i(X)\overline{\otimes}i(Y) = \Colim\bigl(\yo \times \yo \downarrow (i \times i)(X,Y) \xrightarrow{\pi} C \times C \xrightarrow{\otimes \circ (\hat{\yo} \times \hat{\yo})} \cat{C}\bigr)$$
We see from equation (2) in \cref{lemma:comma_cat_isos} that $X \tilde{\otimes} Y \iso i(X) \overline{\otimes} i(Y)$, and thus $\tilde{\otimes} \iso \overline{\otimes} \circ (i \times i)$ as desired. 

For part (\ref{eq:bar-nothing}) it suffices to show that $\overline{\otimes}$ preserves colimits in each variable.
Indeed, since $\otimes \circ (L \times L)$ and $\overline{\otimes}$ agree on representables, and both preserve colimits, we get that $\overline{\otimes} \iso \otimes \circ (L \times L)$. 

To do this, we show that $\overline{\otimes}$ is a Kan extension in each variable.
We claim that for any $X, Y$ in $\hat{C}$:

\begin{align}
    X \overline{\otimes} -  \iso \lan_\yo\bigl((X \overline{\otimes} -) \circ \hat{\yo}\bigr) \\
    - \overline{\otimes} Y  \iso \lan_\yo\bigl((- \overline{\otimes} Y) \circ \hat{\yo}\bigr) 
\end{align}
We show (1) and the proof of (2) is identical. 
Consider the pointwise formula:
\begin{align*}
    X \overline{\otimes}Y & \iso \Colim\bigl(\yo \times \yo \downarrow (X,Y) \to C \times C \xrightarrow{\otimes \circ (\hat{\yo} \times \hat{\yo})}  \cat{C}\bigr) \\
    & \iso \Colim\bigl(\yo \downarrow X \times \yo \downarrow Y \to C \times C \xrightarrow{\otimes \circ (\hat{\yo} \times \hat{\yo})}  \cat{C}\bigr) && \text{By equation (3) of \cref{lemma:comma_cat_isos}}\\
    & \iso \Colim\bigl(\yo \downarrow Y \xrightarrow{\pi} C \xrightarrow{(X\overline{\otimes}- )\circ \hat{\yo}}\cat{C}\bigr) && \text{Since colimits commute with colimits}
\end{align*}
But this is exactly the formula for $\lan_\yo\bigl((X \overline{\otimes} -)(Y) \circ \hat{\yo}\bigr)$. 
Thus $\overline{\otimes}$ is a Kan extension in each variable.
Since $\yo$ is the free cocompletion, we get that $\overline{\otimes}$ preserves colimits in each variable, completing the proof.
\end{proof}

\section{The category of graphs} \label{sec:graphs}
In this section, we introduce the category of graphs that we will be working with.
While there are many different flavors of graphs, here, we work specifically with undirected simple graphs without loops.
We introduce this category as a subcategory of a presheaf category we shall now define.

Let $\mathbb{G}$ be the category generated by the diagram
\[ \begin{tikzcd}
    V \ar[r, yshift=1ex, "s"] \ar[r, yshift=-1ex, "t"'] & \ar[l, "r" description] E \ar[loop right, "\sigma"]
\end{tikzcd} \]
subject to the identities
\[ \begin{array}{l l}
    rs = rt = \id[V] & \sigma^2 = \id[E] \\
    \sigma s = t & \sigma t = s \\
    r \sigma = r.
\end{array} \]

We call the functor category $\Set^{\mathbb{G}^\op}$ the category of \emph{marked multigraphs}.

For $X \in \Set^{\mathbb{G}^\op}$, we write $X_V$ and $X_E$ for the sets $X(V)$ and $X(E)$, respectively.
Explicitly, such a functor consists of sets $X_V$ and $X_E$ together with the following functions between them
\[ \begin{tikzcd}
    X_V \ar[r, "r" description]  &  X_E \ar[l, yshift=1ex, "s"'] \ar[l, yshift=-1ex, "t"] \ar[loop right, "\sigma"]
\end{tikzcd} \]
subject to the dual versions of identities in $\mathbb{G}$. 
The category $\Set^{\mathbb{G}^\op}$ can be viewed as the category of undirected multigraphs such that each vertex has a specified loop.
We can define the category of undirected simple graphs with no loops as a subcategory of $\Set^{\mathbb{G}^\op}$: 

\begin{definition} \label{def:graph}
    The category $\Graph$ is defined to be the full subcategory of $\Set^{\mathbb{G}^\op}$ spanned by elements $X \in \Set^{\mathbb{G}^\op}$ such that the map $(s, t) \colon X_E \to X_V \times X_V$ is a monomorphism.
\end{definition}

For $X \in \Graph$, we see that $X_V$ represents the vertex set and $X_E$ represents the edge set. 
The map $(s, t) \colon X_E \to X_V \times X_V$ assigns to each edge a source and a target vertex. The monomorphism condition ensures that for every vertex pair $(v,w)$ there exists at most one edge from $v$ to $w$. 
The map $\sigma$ ensures that whenever there is an edge from $v$ to $w$, there is also an edge from $w$ to $v$, hence making our graphs undirected. 
The map $r$ sends each vertex to a loop on that vertex, meaning every vertex has exactly one loop.\footnote{In some literature on graph theory, edges from a fixed vertex to itself satisfying $\sigma l = l$ might be called \emph{semi-edges}, while those with $\sigma l \neq l$ are called \emph{loops}.
According to that terminology, which we do not follow, our graphs would have unique semi-edges and no loops.} 
This is logically the same as treating the graphs as having no loops, however the mandatory loop ensures that a morphism in this category is precisely a graph map, as it is usually defined:

If $X, X'$ are graphs in $\Graph$, and $f \colon  X \to X'$ is a natural transformation, the naturality condition means that if there is an edge between vertices $v$ and $w$ in $X$, there must be an edge between $f(v)$ and $f(w)$ in $X'$. 
Since each vertex has a loop, this is equivalent to the condition that either $f(v)$ and $f(w)$ are different vertices connected by an edge, or they are the same vertex. 
Thus morphisms in this category are maps $f \colon  X_{V} \to  X'_{V}$, such that if $v$ is connected to $w$ in $X$, either $f(v)$ is connected to $f(w)$ or $f(v) = f(w)$.

We write $I_n$ to represent the graph in this category with $n + 1$ vertices connected in a line, as seen in Figure 1.

\begin{figure}[h] \label{I_4 ex}
  \centering
  \begin{tikzpicture}
    \foreach \i in {0,...,4} {
      \node[circle, draw] (\i) at (\i,0) {\i};
      \ifnum\i>0
        \pgfmathtruncatemacro{\j}{\i-1}
        \draw (\j) -- (\i);
      \fi
    }
  \end{tikzpicture}
  \caption{The graph $I_4$}
  \label{fig:graph}
\end{figure}
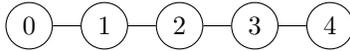

We also write $C_n$ to represent the cyclic graph on $n$ vertices, and $K_n$ to represent the complete graph on $n$ vertices.

Note that by this definition, we are in the same case as with \cref{Prop: monoidal as kan extensions}, as $\Graph$ is a reflective subcategory of $\Set^{\mathbb{G}^\op}$. 
The reflector in this case takes a marked multigraph to the graph obtained by removing instances of multiple edges between two vertices, and replacing them with only a single edge.

The setup also requires that the Yoneda embedding factors through $\Graph$. 
In this case, $\yo \colon \mathbb{G} \to \Set^{\mathbb{G}^\op}$
takes an object $X \in \mathbb{G}$ (either $V$ or $E$) and sends it to the functor $\mathbb{G}(-,X) \colon \mathbb{G}^\op \to \Set$. 

Thus the graph $\yo(V)$ has a vertex set $\mathbb{G}(V,V) = \{ \id \}$ and edge set $\mathbb{G}(E,V) = \{r\}$. 
There are precomposition maps $s^*, t^* \colon \mathbb{G}(E,V) \to \mathbb{G}(V,V)$ sending $r$ to $rs = \id$ and $rt = \id$ respectively. Thus the edge $r$ is from $\id$ to itself. Hence $\yo(V)$ is the graph with a single vertex and a unique loop, in other words, $\yo(V) = I_0$. 

The graph $\yo(E)$ has vertex set $\mathbb{G}(V,E) = \{s,t\}$ and edge set $\mathbb{G}(E,E) = \{\id,\sigma,sr,tr\}$. 
Again, we have precomposition maps $s^*,t^* \colon \mathbb{G}(E,E) \to \mathbb{G}(V,E)$. 
Here $s^*$ takes $\id$ to $s$ and $t^*$ takes $\id$ to $t$, so $\id$ is an edge from $s$ to $t$. Moreover, $s^*$ takes $\sigma$ to $\sigma s=t$ and $t^*$ takes $\sigma$ to $\sigma t = s$. Thus $\sigma$ is an edge in the opposite direction, from $t$ to $s$. 
As stated earlier, pairs of edges like this are treated as one undirected edge. The maps $s^*$ and $t^*$ both map $sr$ to $srs=srt=s$ and map $tr$ to $trs=trt=t$. Thus $sr$ is a loop on the vertex $s$ and $tr$ is a loop on the vertex $t$. In this way, $\yo(E)=I_1$.

We see that as desired, $\yo$ does factor through $\Graph$. This is summarized in the following figure and proposition.

\begin{figure}[h]
    \centering
    \begin{tikzpicture}
    \node[circle, draw] (A) at (0,0) {\(\id\)};
    \draw[->] (A) to [loop above] node {\(r\)} (A);
\end{tikzpicture}
\hspace{2cm} 
\begin{tikzpicture}
    \node[circle, draw] (s) at (0,0) {\(s\)};
    \node[circle, draw] (t) at (0,2) {\(t\)};
    
    \draw[->] (s) to[bend left] node[left] {\(\id\)} (t);
    \draw[->] (t) to[bend left] node[right] {\(\sigma\)} (s);
    \draw[->] (t) to [loop above] node {\(tr\)} (t);
    \draw[->] (s) to [loop below] node {\(sr\)} (s);
\end{tikzpicture}
    \caption{The graphs $I_0$ and $I_1$ as the image of the Yoneda embedding}
    \label{fig:yoneda of graphs}
\end{figure}

\begin{proposition}
    The Yoneda embedding $\yo \colon  \mathbb{G} \to \Set^{\mathbb{G}^\op}$ sends $V$ to $I_0$ and $E$ to $I_1$. \qed   
\end{proposition}

With the category $\Graph$ defined, we now state a corollary of \cref{prop:monoidal_unit} regarding possible units for closed symmetric monoidal structures:

\begin{corollary} \label{cor: graph_unit} \leavevmode
  \begin{enumerate}
      \item The unit of any monoidal structure on $\Graph$ is either the empty graph $\varnothing$ or the graph with a single vertex $I_0$.
      \item The unit of any closed monoidal structure on $\Graph$ is $I_0$ and hence it is semi-cartesian.
  \end{enumerate}
\end{corollary}

\begin{proof}
    This follows from \cref{prop:monoidal_unit}. In the category of graphs, $I_0$ is the terminal object and it has exactly two subobjects: $\varnothing$ and $I_0$.
    Moreover, $\varnothing$ is the initial object and hence it must be preserved by the functor $- \otimes X$ if the monoidal structure were closed.
    If $\varnothing$ is the unit, we get $X \iso \varnothing \otimes X \iso \varnothing$.
\end{proof}

We finally define two products of graphs that are of interest to us in this paper, in a graph theoretical sense:

\begin{definition} \label{def:graph products}
Let $X$, $X'$ be graphs in $\Graph$. 
\begin{itemize}
    \item The \emph{box product}, denoted $X \square X'$, is the graph with the vertex set $X_{V} \times X'_{V}$, and vertices $(v,v')$, $(w,w')$ have an edge between them if and only if either $v \sim w$ and $v' = w'$, or $v = w$ and $v' \sim w'$.
    \item The \emph{categorical product}, denoted $X \boxtimes X'$, is the graph with the vertex set $X_{V} \times X'_{V}$, and vertices $(v,v')$, $(w,w')$ have an edge between them if and only if either $v = w$ or $v \sim w$, and either $v' = w'$ or $v' \sim w'$
\end{itemize}
\end{definition}
Note that here we are writing $v \sim w$ to represent an edge between two vertices $v$ and $w$.

With these products defined, we have the following proposition:
\begin{proposition}
    Both the box and categorical products define closed symmetric monoidal structures on $\Graph$
\end{proposition}

\begin{proof}
    It is easy to check that both products are monoidal with unit $I_0$, and are each symmetric.
    All that remains is closure:

    For the box product, the right adjoint to $X \square -$ is given by $\hom^\square(X,X')_V = \Graph(X,X')$, the set of graph maps from $X$ to $X'$. 
    There is an edge between distinct maps $f$ and $g$ if for all $v \in X_V$, either $f(v) = g(v)$ or $f(v) \sim g(v)$.

    For the categorical product, the right adjoint to $X \boxtimes -$ is given by $\hom^\boxtimes(X,X')_V = \Graph(X,X')$, the set of graph maps from $X$ to $X'$. 
    There is an edge between distinct maps $f$ and $g$ if whenever two vertices $v$ and $w$ in $X$ are connected by an edge, $f(v) \sim g(w)$ in $X'$.    
\end{proof}

There are also several other notions of products of graphs, all with the vertex set given by $X_V \times Y_V$, which are summarized in the table below:

\begin{center}
\begin{tabular}{||c c c c c||} 
 \hline
 \textbf{Name} & \textbf{Edge Condition} & \textbf{Monoidal} & \textbf{Symmetric} & \textbf{Closed}  \\ [0.5ex] 
 \hline\hline
 Categorical & $(v=w \lor v \sim w) \land (v' = w' \lor v'\sim w')$ & Yes & Yes & Yes\\ 
 \hline
  Box & $(v=w \land v' \sim w') \lor (v \sim w \land v'=w')$ & Yes & Yes & Yes\\ 
 \hline
  Tensor & $v \sim w \land v' \sim w'$ & No & N/A & N/A\\ 
 \hline
  Lexicographical & $(v\sim w) \lor (v = w \land v'\sim w')$ & Yes & No & No\\ 
  \hline
  Conormal & $v \sim w \lor v' \sim w'$ & Yes & Yes & No\\ 
 \hline
  Modular & $(v \sim w \land v' \sim w') \lor (v \nsim w \land v' \nsim w') $ & No & N/A & N/A\\ [1ex] 
 \hline
\end{tabular}
\end{center}
\section{Products via the left Kan extension} \label{sec:Kan ext}

\newcommand{\pair}[1]{\textbf{#1}}
To classify closed symmetric monoidal structures on $\Graph$, we mimic the proof of \cref{ex: monoidal in set}, relying on \cref{Prop: monoidal as kan extensions}. 
As such, we need to examine possible functors $F \colon \mathbb{G} \times \mathbb{G} \to \Graph$ such that the left Kan extension of $F$ along $\hat{\yo} \times \hat{\yo}$ is closed symmetric monoidal.
In this section, we establish preliminary results that will be used in the following section to provide the desired classification.

Our analysis depends primarily on the pointwise formula for the left Kan extension, which for graphs $X$ and $X'$ is given by:
$$(\lan_{\hat{\yo} \times \hat{\yo}}F)(X,X') = \Colim\bigl(\hat{\yo} \times \hat{\yo} \downarrow (X,X') \xrightarrow{\pi^{(X,X')}} \mathbb{G} \times \mathbb{G} \xrightarrow{F} \Graph\bigr)\text{.}$$
Recall that a colimit in the category of graphs is computed by taking the disjoint union of all vertices in the diagram, and quotienting by an equivalence relation $\sim$. 
Here $\sim$ is generated by the condition that a vertex $v \sim w$ if there exists a map $f$ in the colimit diagram such that $f(v) = w$. 
There is an edge between equivalence classes $[v]$ and $[w]$ if there exists a $a \in [v]$ and $b \in [w]$ such that there is an edge between $a$ and $b$.

For notational convenience, going forward, we write $\pi$ for $\pi^{(X,X')}$ and $\yo$ for $\hat{\yo}$.

 We first analyze the category $\yo \times \yo \downarrow (X, X')$.
 The objects in this category are of the form

 \[\Bigl((A,A') \in \bbG^2, (f, f') \from \bigl(\yo(A) , \yo(A')\bigr) \to (X, X')\Bigr)\text{.}\] 
 We see that if $(A,A')=(E,E)$, this corresponds to a choice of an edge in $X$ and an edge in $X'$. 
 If $(A,A') = (V,V)$ this corresponds to a choice of a vertex in each graph, and if $(A,A') = (V,E)$ or $(E,V)$ this corresponds to a choice of a vertex in one graph and edge in another. 
 A morphism $H \colon \bigl((A,A'), (f,f')\bigr) \to \bigl((B,B'), (g,g')\bigr)$ is a morphism in $\mathbb{G} \times \mathbb{G}$ such that the diagram
 \begin{center}
\begin{tikzcd}
{\yo \times \yo(A,A')} \arrow[r, "{(f,f')}"] \arrow[d, "\yo \times \yo(H)"', shift left] & {(G, G')} \\
{\yo \times \yo(B,B')} \arrow[ru, "{(g,g')}"']                         &            
\end{tikzcd}
\end{center}
commutes. 

To compute the colimit, one might want to work with a subcategory of  $\yo \times \yo \downarrow (X, X')$ which we now define.

\begin{definition} \label{def:subcategory of comma cat}
Let $X, X'$ be graphs.
\begin{enumerate}
    \item The \emph{edge subcategory} $\bigl(\yo \times \yo \downarrow (X,X')\bigr)_{(E,E)}$ of $\yo \times \yo \downarrow (X, X')$ is the full subcategory with objects of the form $\bigl((E,E),(f,f') \colon (I_1,I_1) \to (X,X')\bigr)$.
    \item The \emph{edge subcategory} $(\yo \downarrow X)_{E}$ of $\yo \downarrow X$ is the full subcategory with objects of the form $(E,f \colon \yo(E) \to X)$.
\end{enumerate}
\end{definition}

\begin{lemma} \label{yoE_final}
    The inclusion of the edge subcategory $i \colon  (\yo \downarrow X)_E \hookrightarrow \yo \downarrow X$ is final for any graph $X$.
\end{lemma}

\begin{proof}
We want to show that for every object $c \in \yo \downarrow X$ the comma category $(c\downarrow i)$ is connected. 
We know that $c$ is either of the form $\bigl(V,f \colon \yo(V) \to X\bigr)$ which is precisely a choice of a vertex in X, or $\bigl(E, f \colon \yo(E) \to X\bigr)$ which is precisely a choice of edge (possibly a loop) in $X$, where one vertex is identified as the $s$ vertex and the other as $t$. 
Recall that objects in $(\yo \downarrow X)_E$ are of the form $\bigl(E, e \colon  \yo(E) \to X\bigr)$ which again is simply a choice of an edge. 
Thus, if $c$ is of the form $\bigl(A,c \colon  \yo(A) \to X\bigr)$ objects in $(c\downarrow i)$ are a choice of an edge in $X$ (with source and target identified), along with a morphism $k \colon  A \to E$ in $\mathbb{G}$ such that the following diagram commutes: 
\begin{center}
\begin{tikzcd}
\yo(A) \arrow[r, "c"] \arrow[d, "\yo(k)"'] & X \\
\yo(E) \arrow[ru, "e"']                                   &  
\end{tikzcd}
\end{center}

We thus have two cases, as $c$ can either be a choice of edge or vertex. We show that $c \downarrow i$ is connected in either case.

\textbf{Case 1:} Suppose first that $c$ corresponds to a choice of vertex. Moreover suppose $(c,e_1, f_1), (c,e_2,f_2) \in (c\downarrow i)$, where $f_1, f_2 \in \mathbb{G}(V,E)$.  
Here $c$ is of the form $\bigl(V,c \colon \yo(V) \to G\bigr)$, and $e_1, e_2$ are of the form $\bigl(E,e_1 \colon  \yo(E) \to X\bigr)$ and $\bigl(E,e_2 \colon  \yo(E) \to X\bigr)$.
Since both $f_1$ and $f_2$ make the above diagram commute, we have that $$e_1 \yo f_1 = c = e_2 \yo f_2\text{.}$$
Note that $\yo(V)$ is $I_0$ and $\yo(E)$ is $I_1$, with one vertex labelled $s$ and the other $t$. 
Thus $\yo f_1$ corresponds to the map taking the single vertex in $\yo(V)$ to the $f_1$ vertex of $\yo(E)$, since $f_1$ is either $s$ or $t$. 
The map $f_2$ also acts identically. 
Thus, the above equation simply means that the $f_1$ vertex of the edge chosen by $e_1$ must equal the $f_2$ vertex of the edge chosen by $e_2$, which both must be the vertex chosen by $c$.
For simplicity, we refer to the vertex in the image of $c$ simply as $c$. 

Now, consider the element $(c,e,s)$ of $(c\downarrow i)$, where $e$ corresponds to choosing the loop on $c$ as the edge, and $s \colon  V \to E$ (note that $t$ also works). 
We claim that the diagram:

\begin{center}
\begin{tikzcd}
        & V \arrow[ld, "f_1"'] \arrow[d, "s"] \arrow[rd, "f_2"] &                          \\
E \arrow[rd, "e_1"'] & E \arrow[l, "f_1r"'] \arrow[r, "f_2r"] \arrow[d, "e"] & E \arrow[ld, "e_2"] \\
                          & X                                                          &                         
\end{tikzcd}
\end{center}
commutes. Clearly the top triangles commute since $f_1rs = f_1\id[V] = f_1$ and similarly $f_2rs = f_2$. But we also know that $f_1r$ maps both vertices to the $f_1$ vertex (recall $f_1 = s$ or $t$) which is by definition $c$. Then $e_1f_1r$ is equal to choosing the loop on vertex $c$, so $e_1f_1r = e$. Similarly $e_2f_2r$ = $e$. This shows that the two objects $(c,e_1,f_1)$ and $(c,e_2,f_2)$ are connected through $(c,e,s)$.

\textbf{Case 2}: Suppose $c$ corresponds to a choice of edge. Let $(c,e_1,f_1)$ and $(c,e_2,f_2)$ be as above, where $f_1, f_2 \in \mathbb{G}(E,E)$. Then if $e$ represents the map: $\yo(E) \to G$ corresponding to the choice of edge $c$ in $X$, clearly the diagram:
\begin{center}
\begin{tikzcd}
                     & E \arrow[ld, "f_1"'] \arrow[d, "{\id[E]}"] \arrow[rd, "f_2"] &                     \\
E \arrow[rd, "e_1"'] & E \arrow[l, "f_1"'] \arrow[r, "f_2"] \arrow[d, "e"]          & E \arrow[ld, "e_2"] \\
                     & X                                                            &                    
\end{tikzcd}
\end{center}
commutes, since by definition $e_1f_1 = e$ and $e_2f_2 = e$ since $(c,e_1,f_1)$ and $(c,e_2,f_2)$ are objects in the comma category. So both objects are connected through $(c,e,\id[E])$.

Thus, for any $c$, the comma category $(c \downarrow i)$ is connected, so $i$ is final. 
\end{proof}
 
 We now have the following corollaries: 

\begin{corollary} \label{cor:yo_E final}

Let $X, X' \in \Graph$. Then the inclusion $i \colon  \bigl(\yo \times \yo \downarrow (X,X')\bigr)_{(E,E)} \hookrightarrow \yo \times \yo \downarrow (X,X')$ is final.
\end{corollary}

\begin{proof}
This follows by \cref{yoE_final} and the fact that the product of final functors is final.
\end{proof}

\begin{corollary} \label{cor:colimit formula}
The monoidal structure $\otimes$ is completely determined by $F(E,E)$.
\end{corollary}

\begin{proof}
    Let $F \colon \mathbb{G} \times \mathbb{G} \to \Graph$ be such that $(\lan_{\hat{\yo} \times \hat{\yo}}F)$ is a monoidal product $\otimes$ on $\Graph$.
    By the pointwise formula for left Kan extensions, $X \otimes X'$ is given by $\Colim\bigl(\hat{\yo} \times \hat{\yo} \downarrow (X,X') \xrightarrow{\pi^{(X,X')}} \mathbb{G} \times \mathbb{G} \xrightarrow{F} \Graph\bigr)$.
    Since $\bigl(\yo \times \yo \downarrow (X,X')\bigr)_{(E,E)} \hookrightarrow \yo \times \yo \downarrow (X,X')$ is final, we could instead compute the colimit over $\bigl(\yo \times \yo \downarrow (X,X')\bigr)_{(E,E)}$ for which only the value of $F$ at $(E, E)$ is required.
\end{proof} 

\section{Main Theorem}\label{sec:main-thm}

In this section, we prove our main theorem (\cref{main th}), characterizing closed symmetric monoidal structures on the category of graphs.

For the remainder of the paper, we impose the following assumption: 
\begin{assumption} \label{assumption4}
    Fix a closed symmetric monoidal structure $\otimes$ on $\Graph$, and set $F_\otimes := \otimes \circ (\yo \times \yo)$.
\end{assumption}

We claim that the only two such functors $\otimes$ are the box and categorical products.
In broad strokes, the argument can be summarized as:
\begin{enumerate}
    \item $F_\otimes(E,E)$ cannot be a graph with more than four vertices (\cref{at most 4}).
    \item $F_\otimes(E,E)$ cannot be a graph with less than four vertices (\cref{cor:exactly 4}).
    \item If $F_\otimes(E,E)$ has exactly four vertices, and the resulting product is closed symmetric monoidal, then the resulting product is either the box or categorical product (\cref{TH:monoidal structures}).
\end{enumerate}

Recall that, by \cref{Prop: monoidal as kan extensions}, we know that $\otimes \iso \lan_{\yo \times \yo} F_\otimes$.
By \cref{cor:colimit formula}, we only need to consider the action of $F_\otimes$ on $(E,E)$.

\begin{lemma} \label{at most 4}
    The product $I_1 \otimes I_1$ has at most four vertices.
\end{lemma}

\begin{proof}
Consider the following pushout square in $\Graph$:
    \[ \begin{tikzcd}[column sep = large]
        I_0 \sqcup I_0 
            \arrow[r,tail,"\partial"] \arrow[d,tail,swap,"\partial"] & 
        I_1 
            \arrow[d,"\id"] \\
        I_1 
            \arrow[r,"\id"] & 
        I_1
    \end{tikzcd} \]
Applying $I_1 \otimes -$ and using the fact that the monoidal structure is closed and has unit $I_0$, we get a pushout square
    \[ \begin{tikzcd}[column sep = large]
        I_1 \sqcup I_1 
            \arrow[r,"\partial \otimes \id"] \arrow[d,swap,"\partial \otimes \id"] & 
        I_1 \otimes I_1 
            \arrow[d,"\id"] \\
        I_1 \otimes I_1 
            \arrow[r,"\id"] & 
        I_1 \otimes I_1
    \end{tikzcd} \]
(Note that we used here that functors preserve identity morphisms.)
But a square of this form is a pushout if and only if $\partial \otimes \id \colon I_1 \sqcup I_1 \to I_1 \otimes I_1$ is an epimorphism, and hence in particular surjective on vertices.
As its domain has four vertices, the codomain cannot have more than four vertices.
\end{proof}

The task of showing that $F_\otimes(E,E)$ must have at least four vertices is significantly more involved. 
We begin with the following proposition and definition, which are used throughout the remainder of this section:

\begin{proposition} \label{image of I_0}
The functor $F_\otimes$ must satisfy: $F_\otimes(V,V)=I_0$ and $F_\otimes(V,E)=F_\otimes(E,V)=I_1$.
\end{proposition}
\begin{proof}
    We have:
\begin{align*}
    F_\otimes(V,V) &= \otimes \circ \yo \times \yo (V,V) \\
    &= I_0 \otimes I_0 \\
    &= I_0
\end{align*}
since $\otimes$ is monoidal and $I_0$ must be the unit by \cref{cor: graph_unit}.
The other parts are analogous.
\end{proof}

\begin{definition} \label{def:vertex representation}
We define \textit{labelled} vertices as follows:
\begin{itemize}
\item[(1)] We say a vertex $w$ in $F_\otimes(E,E)$ is \emph{labelled} $(a,a')$ for $a, a' \in \{s,t\}$ if $F_\otimes(a,a')(0) = w$, where $0$ is the unique vertex of $I_0$.
\item[(2)] We say a vertex $w$ in $F_\otimes(E,V)$ is \emph{labelled} $a$ for $a \in \{s,t\}$ if $F_\otimes(a,\id)(0) = w$ (where, once again, $0$ is the unique vertex of $I_0$), and the vertices in $F_\otimes(V,E)$ are labelled analogously.
\end{itemize}
\end{definition}

\begin{example}
    Consider the closed symmetric monoidal product $\square$. 
    If $F_\square$ is such that $\square \iso \lan_{\yo \times \yo} F_\square$, then by \cref{Prop: monoidal as kan extensions} we must have $F_\square (V,V)=I_0$, $F_\square (E,V)=F_\square (V,E)=I_1$ and $F _\square(E,E)= C_4$. We also know $F_\square$ also acts on the morphisms $(s,s)$, $(s,t)$, $(t,s)$, and $(t,t)$ by sending them to distinct graph maps from $I_0$ to $C_4$. 
    So, in this case, the four vertices of $F_\square (E,E)$ are each are labelled one of $(s,s)$, $(s,t)$, $(t,s)$, and $(t,t)$, based on where $F_\square (s,s)$, $F_\square (s,t)$, $F_\square (t,s)$, and $F_\square (t,t)$ map $I_0$, as shown below.
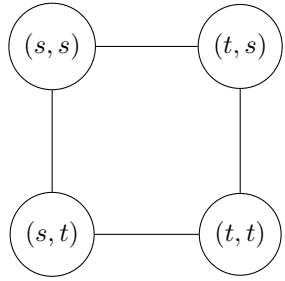
\begin{center}
\begin{tikzpicture}[node distance={25mm}, main/.style = {draw, circle}] 
\node[main] (1) {$(s,s)$}; 
\node[main] (2) [right of=1] {$(t,s)$};
\node[main] (3) [below of=2] {$(t,t)$};
\node[main] (4) [left of=3] {$(s,t)$};

\draw (1) -- (2);
\draw (2) -- (3);
\draw (3) -- (4);
\draw (4) -- (1);

\end{tikzpicture} 
\end{center}
\end{example}

Many of the following proofs also rely on using the pointwise formula and analyzing the colimit diagrams.
As the diagrams can get rather complex, before proceeding it makes sense to first gain some familiarity with what these colimits can look like.
We know that objects in $\bigl(\yo \times \yo \downarrow (X,X')\bigr)_{(E,E)}$ are of the form $\bigl((E,E), (f,f') \colon  (I_1, I_1) \to (X,X')\bigr)$.
For simplicity, we refer to such an object as $(f,f')$.
\begin{example} \label{comma cat ex}
Consider the category $\bigl(\yo \times \yo \downarrow (I_0,I_1)\bigr)_{(E,E)}$.
This category has four objects, each defined by a map $(I_1,I_1) \to (I_0,I_1)$. 
As there is only one map from $I_1$ to $I_0$, these objects are completely determined by maps $I_1$ to $I_1$. 
Call the four objects $\textsf{Idty}$, $\textsf{Flip}$, $\textsf{Const}_s$, and $\textsf{Const}_t$, based how they map $I_1$.
Note that we label the vertices of $I_1$ as $s$ and $t$.
The resulting diagram and morphisms can be seen below:

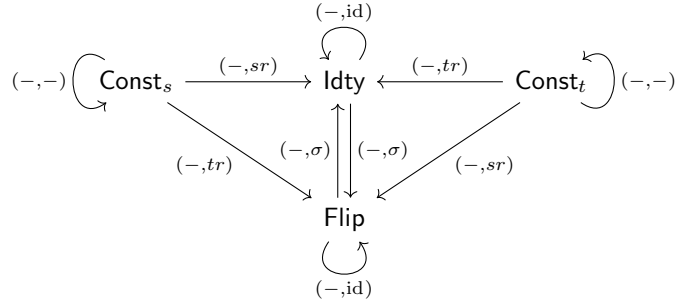
\begin{center}
\begin{tikzcd}
\textsf{Const}_s \arrow[rr, "{(-,sr)}"] \arrow[rrdd, "{(-,tr)}"'] \arrow["{(-,-)}"', loop, distance=2em, in=215, out=145] &  & \textsf{Idty} \arrow[dd, "{(-,\sigma)}", shift left] \arrow["{(-,\id)}"', loop, distance=2em, in=125, out=55]  &  & \textsf{Const}_t \arrow[ll, "{(-,tr)}"'] \arrow[lldd, "{(-,sr)}"] \arrow["{(-,-)}"', loop, distance=2em, in=35, out=325] \\
                                                                                                           &  &                                                                                                    &  &                                                                                                           \\
                                                                                                           &  & \textsf{Flip} \arrow[uu, "{(-,\sigma)}", shift left] \arrow["{(-,\id)}"', loop, distance=2em, in=305, out=235] &  &                                                                                                          
\end{tikzcd}                                     
\end{center}
Here each arrow represents multiple morphisms, as `-' is used to represent any morphism in $\mathbb{G}(E,E)$.

The objects represent maps as follows:
\begin{itemize}
    \item The object $\textsf{Idty}$ corresponds to the identity map, sending the vertex $s$ to $s$ and $t$ to $t$.
    \item The object $\textsf{Flip}$ corresponds to the map sending the $s$ vertex to $t$ and the $t$ vertex to $s$.
    \item The object $\textsf{Const}_s$ corresponds to the constant map sending both vertices to the $s$ vertex of $I_1$.
    \item The object $\textsf{Const}_t$ corresponds to the constant map sending both vertices to the $s$ vertex of $I_1$.
\end{itemize}

If $F_\otimes$ is a functor $F_\otimes \colon  \mathbb{G} \times \mathbb{G} \to \Graph$, each of these objects get sent to graphs in the colimit diagram by $F_\otimes \circ \pi$, and the morphisms between them to graph maps. 

\end{example}

With this out of the way, we now proceed with further restricting the possibilities for $F_\otimes$.
Note that in the following proofs, for simplicity, when referring to maps such as $F_\otimes(sr,\sigma)$, we omit the $F_\otimes$ and simply write $(sr,\sigma)$.

\begin{lemma} \label{double labels}
Suppose $F_\otimes(E,E)$ has a vertex with more than one label. Then either $(s,s) = (t,s)$ and $(s,t) = (t,t)$, $(s,s) = (s,t)$ and $(t,s) = (t,t)$, or both.
\end{lemma}

\begin{proof}
    Suppose that $F_\otimes(E,E)$ contains a double-labelled vertex. 
    Then at least one of the following must hold:
    \begin{enumerate}
        \item $(s,s) = (t,s)$
        \item $(s,t) = (t,t)$
        \item $(s,s) = (s,t)$
        \item  $(t,s) = (t,t)$
        \item $(s,s)=(t,t)$
        \item $(s,t)=(t,s)$
    \end{enumerate}
    We show that (1) if and only if (2), (3) if and only if (4), (5) implies both (1) and (3), and hence also (2) and (4), and finally that (6) implies (5).

    \textbf{(1) $\Longleftrightarrow$ (2):} Suppose $(s,s) = (t,s)$. Then:
    \begin{align*}
        (s,t) &= (\id, \sigma)(s,s)\\
        &=(\id,\sigma)(t,s)\\
        &=(t,t)
    \end{align*}
    
    Now suppose $(t,t)=(s,t)$. Then:
    \begin{align*}
        (t,s)&= (\id,\sigma)(t,t) \\
        &=(\id,\sigma)(s,t) \\
        &=(s,s)
    \end{align*}
    
    \textbf{(3) $\Longleftrightarrow$ (4):} Suppose $(s,s) = (s,t)$. Then:
    \begin{align*}
        (t,s) &= (\sigma, \id)(s,s)\\
        &=(\sigma,\id)(s,t)\\
        &=(t,t)
    \end{align*}

    Now suppose $(t,t)=(t,s)$. Then:
    \begin{align*}
        (s,t) &= (\sigma, \id)(t,t) \\
        &= (\sigma, \id)(t,s) \\
        &= (s,s)
    \end{align*}

    \textbf{(5) $\implies$ (1):} Suppose $(s,s) = (t,t)$. Then:
    \begin{align*}
        (s,s) &= (\id, sr)(s,s)\\
        &=(\id,sr)(t,t)\\
        &=(t,s)
    \end{align*}

    \textbf{(5) $\implies$ (3):} Suppose $(s,s) = (t,t)$. Then:
    \begin{align*}
        (s,s) &= (sr, \id)(s,s)\\
        &=(sr,\id)(t,t)\\
        &=(s,t)
    \end{align*}

    \textbf{(6) $\implies$ (5):} Suppose $(s,t)=(t,s)$. Then:
    \begin{align*}
        (s,s) &= (\sigma,\id)(t,s) \\
        &= (\sigma, \id)(s,t) \\
        &= (t,t)
    \end{align*}

    Thus, if $F_\otimes(E,E)$ has a double label, either $(s,s) = (t,s)$ and $(s,t) = (t,t)$, $(s,s) = (s,t)$ and $(t,s) = (t,t)$, or both.
\end{proof}

\begin{lemma}\label{no double labels}
    The graph $F_\otimes(E,E)$ cannot be such that any vertex has more than one label.
\end{lemma}

\begin{proof}
    Suppose $F_\otimes(E,E)$ contains a double labelled vertex. 
    We show that $\otimes$ is not closed symmetric monoidal. 
    By \cref{double labels}, we have two possibilities to consider. 
    Either $(s,s) = (t,s)$ and $(s,t) = (t,t)$ or $(s,s) = (s,t)$ and $(t,s) = (t,t)$. 
    For this proof, assume that $(s,s) = (s,t)$ and $(t,s) = (t,t)$. 
    The proof where $(s,s) = (t,s)$ and $(s,t) = (t,t)$ is identical, the coordinates are just reversed. 
    Also note that we are not assuming both are not the case, just that at least  $(s,s) = (s,t)$ and $(t,s) = (t,t)$.
    To show that $\otimes$ is not closed symmetric monoidal, we show that $I_0$ is not the unit, which must be the case if $\otimes$ was closed symmetric monoidal by \cref{cor: graph_unit}. 

    Before we begin, first note that the graph $F_\otimes(E,E)$ may have vertices that are unlabelled. These unlabelled vertices are vertices that are not mapped to by any map of the form $F_\otimes(a,b) \colon I_0 \to F_\otimes(E,E)$, where $a,b \in \{s,t\}$.
    We break the proof into two separate cases, based on how the map $(sr, \id): F_\otimes(E,E) \to F_\otimes(E,E)$ acts on any unlabelled vertices:

    \begin{description}
        \item[Case 1:] The map $(sr, \id)$ is not such that $(sr, \id)(a) = b$ for some (not necessarily distinct) unlabelled vertices $a$ and $b$.
        \item[Case 2:] The map $(sr, \id)$ is such that $(sr, \id)(a) = b$ for some (not necessarily distinct) unlabelled vertices $a$ and $b$.
    \end{description}
    The proof goes roughly as follows: 
    
    In case (1), we consider the product $I_0 \otimes I_1$. 
    We show that $I_0 \otimes I_1 = I_0$, which contradicts unitality. 
    This is roughly because, in the colimit diagram, all vertices get related through the double labelled vertices along with the maps of the form $(sr, \id)$.

    In case (2), we consider the product $I_0 \otimes I_3$. We show that $I_0 \otimes I_3 \neq I_3$, again contradicting unitality. 
    This roughly happens because certain unlabelled vertices do not end up getting related to any other vertices, which causes there to be extra vertices in the colimit. More formally, we have:

    \textbf{Case 1:} Suppose the map $(sr, \id)$ sends any unlabelled vertices in $F_\otimes(E,E)$ to labelled ones (or that there are no unlabelled vertices). 
    Consider the product $I_0 \otimes I_1$. 
    Recall from \cref{comma cat ex} that the category $(\yo \times \yo \downarrow \bigl(I_0,I_1)\bigr)_{(E,E)}$ looks like: 

    
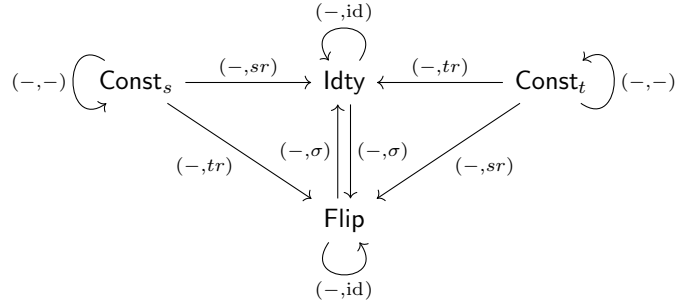
\begin{center}
\begin{tikzcd}
\textsf{Const}_s \arrow[rr, "{(-,sr)}"] \arrow[rrdd, "{(-,tr)}"'] \arrow["{(-,-)}"', loop, distance=2em, in=215, out=145] &  & \textsf{Idty} \arrow[dd, "{(-,\sigma)}", shift left] \arrow["{(-,\id)}"', loop, distance=2em, in=125, out=55]  &  & \textsf{Const}_t \arrow[ll, "{(-,tr)}"'] \arrow[lldd, "{(-,sr)}"] \arrow["{(-,-)}"', loop, distance=2em, in=35, out=325] \\
                                                                                                           &  &                                                                                                    &  &                                                                                                           \\
                                                                                                           &  & \textsf{Flip} \arrow[uu, "{(-,\sigma)}", shift left] \arrow["{(-,\id)}"', loop, distance=2em, in=305, out=235] &  &                                                                                                          
\end{tikzcd}                                     
\end{center}
Our goal is to show that all vertices in the colimit diagram get related, and thus belong to the same equivalence class. 
This would show that the colimit has only one vertex, and is thus $I_0$. 

First, we show that all labelled vertices are in the same equivalence class. 
Let $v$ be a vertex labelled $(s,s)$, and thus also $(s,t)$ in $F_\otimes\bigl(\pi(\mathsf{Idty})\bigr)$. 
Consider the map $(sr, sr) \colon F_\otimes\bigl(\pi(\textsf{Const}_s)\bigr) \to F_\otimes\bigl(\pi(\textsf{Idty})\bigr)$. 
This map is constant with a value of $v$, since for a vertex $w \in F_\otimes\bigl(\pi(\textsf{Const})\bigr)$:
\begin{align*}
    (sr,sr)(w) &= (s,s) \circ (r,r)(w)\\
    &= (s,s)(I_0) \\
    &= v
\end{align*}
since $v$ has a the label $(s,s)$. 
Likewise, there is a constant map $(sr, tr) \colon F_\otimes\bigl(\pi(\textsf{Const}_t)\bigr) \to F_\otimes\bigl(\pi(\textsf{Idty})\bigr)$ which is constant with a value of $v$. 
What we have just shown is that the vertex labelled both $(s,s)$ and $(s,t)$ in $F_\otimes\bigl(\pi(\textsf{Idty})\bigr)$ belongs to the same equivalence class as all vertices in both $F_\otimes\bigl(\pi(\textsf{Const}_s)\bigr)$ and $ F_\otimes\bigl(\pi(\textsf{Const}_t)\bigr)$. 

We can use the same method, however, to show that this is the case for all labelled vertices in $F_\otimes\bigl(\pi(\textsf{Idty})\bigr)$ and $F_\otimes\bigl(\pi(\textsf{Flip})\bigr)$:
\begin{itemize}
    \item For the vertex labelled $(t,s)$ and $(t,t)$ in $F_\otimes\bigl(\pi(\textsf{Idty})\bigr)$, there are constant maps $(tr,sr)$ and $(tr,tr)$ from $F_\otimes\bigl(\pi(\textsf{Const}_s)\bigr)$ and $F_\otimes\bigl(\pi(\textsf{Const}_t)\bigr)$, respectively.
    \item For the vertex labelled $(s,s)$ and $(s,t)$ in $F_\otimes\bigl(\pi(\textsf{Flip})\bigr)$, there are constant maps $(sr,tr)$ and $(sr,sr)$ from $F_\otimes\bigl(\pi(\textsf{Const}_s)\bigr)$ and $F_\otimes\bigl(\pi(\textsf{Const}_t)\bigr)$, respectively.
    \item For the vertex labelled $(t,s)$ and $(t,t)$ in $F_\otimes\bigl(\pi(\textsf{Flip})\bigr)$, there are constant maps $(tr,tr)$ and $(tr,sr)$ from $F_\otimes\bigl(\pi(\textsf{Const}_s)\bigr)$ and $F_\otimes\bigl(\pi(\textsf{Const}_t)\bigr)$, respectively.
\end{itemize}
Thus, all labelled vertices in both $F_\otimes\bigl(\pi(\textsf{Idty})\bigr)$ and $F_\otimes\bigl(\pi(\textsf{Flip})\bigr)$ belong to the same equivalence class as all vertices in both $F_\otimes\bigl(\pi(\textsf{Const}_s)\bigr)$ and $ F_\otimes\bigl(\pi(\textsf{Const}_t)\bigr)$. 

The only vertices that remain are potentially unlabelled vertices in $F_\otimes\bigl(\pi(\textsf{Idty})\bigr)$ and $F_\otimes\bigl(\pi(\textsf{Flip})\bigr)$, if they exist. 
By assumption, however, the map $(sr,\id)$ must map any unlabelled vertex to labelled ones. 
Since there is a map $(sr,\id)$ from both $F_\otimes\bigl(\pi(\textsf{Idty})\bigr)$ to itself and $F_\otimes\bigl(\pi(\textsf{Flip})\bigr)$ to itself, we have that any unlabelled vertices in these two graphs must also belong to the same equivalence class as all the labelled vertices. Thus, since there is only one equivalence class of vertices in the colimit, we get $I_0 \otimes I_1 = I_0$, contradicting unitality, so $\otimes$ cannot be closed symmetric monoidal.

\textbf{Case 2:} Suppose $(sr,\id)$ is such that $(sr, \id)(a) = b$ for some (not necessarily distinct) unlabelled vertices $a$ and $b$. 
In this case, the trick of considering $I_0 \otimes I_1$ does not work. 
Instead, we must consider $I_0 \otimes I_3$. 
This is more complicated, however, and requires some setup before proceeding. 

We first note that $(sr,\id)(b) = b$. This is because:
\begin{align*}
    (sr, \id)(b) &= (sr, \id)\bigl((sr,\id)(a)\bigr)\\
    &=(sr, \id) \circ (sr, \id)(a)\\
    &= (sr, \id)(a)\\
    &= b
\end{align*}
With this in mind, we construct a set of vertices in $F_\otimes(E,E)$ which we call $\mathsf{Inv}^\sigma$. 
The purpose of this construction, roughly, is to have a set of vertices that do not get related to other vertices in the colimit.

First, add $a$ and $b$ into $\mathsf{Inv}^\sigma$. 
Next, if there exists a vertex $c$ such that $(sr, \id)(c) \in \mathsf{Inv}^\sigma$, add $c$ to $\mathsf{Inv}^\sigma$. 
Repeat this step until no such $c$ exists.
Finally, for every vertex $a \in \mathsf{Inv}^\sigma$, add in $(\id, \sigma)(a)$ and $(\sigma,\id)(a)$. 
Note that we only need to do this once, since $(\id, \sigma)^2 = (\id, \id) = (\sigma,\id)^2$.

Note that $\mathsf{Inv}^\sigma$ cannot contain any labelled vertices. 
This is because none of $(sr, \id)$, $(\sigma \id)$, or $(\id,\sigma)$ can map a labelled vertex to an unlabelled one.
To see this, suppose $v$ is labelled $(x,y)$ for $x,y \in \{s,t\}$. 
This means that $(x,y)(I_0) = v$. Then by functoriality: 
\begin{align*}
    (sr, \id)(v) &= (sr,\id) \circ (x,y)(I_0)\\ 
    &= (srx, \id y)(I_0)\\ 
    &= (s,y)(I_0)
\end{align*}
which is another labelled vertex by definition. 
An identical argument shows that $(\id, \sigma)(v)$ and $(\sigma, \id)(v)$ must also be labelled.
So $\mathsf{Inv}^\sigma$ cannot contain any labelled vertex. 
We thus have a set of unlabelled vertices $\mathsf{Inv}^\sigma$, such that $(sr, \id)(\mathsf{Inv}^\sigma) \subseteq \mathsf{Inv}^\sigma$, $(\id,\sigma)(\mathsf{Inv}^\sigma) \subseteq \mathsf{Inv}^\sigma$, $(\sigma, \id)(\mathsf{Inv}^\sigma) \subseteq \mathsf{Inv}^\sigma$, and for any vertex $v \notin \mathsf{Inv}^\sigma$, $(sr, \id)(v) \notin \mathsf{Inv}^\sigma$. 
In order to show that the vertices in $\mathsf{Inv}^\sigma$ don't get related to other vertices in the colimit, we have the following claim:

\textbf{Claim:}
\begin{itemize}
    \item[(1)] For any map $f$ of the form $(-,\id)$, $f(\mathsf{Inv}^\sigma) \subseteq \mathsf{Inv}^\sigma$
    \item[(2)] For any map $f$ of the form $(-, \sigma)$, $f(\mathsf{Inv}^\sigma) \subseteq \mathsf{Inv}^\sigma $
    \item[(3)] For any map $f$ of the form $(-,\id)$, and any any vertex $v \notin \mathsf{Inv}^\sigma$, $f(w) \notin \mathsf{Inv}^\sigma$
    \item[(4)] For any map $f$ of the form $(-, \sigma)$, and any any vertex $v \notin \mathsf{Inv}^\sigma$, $f(w) \notin \mathsf{Inv}^\sigma$
\end{itemize}

\textbf{Proof of claim:}
For (1), $(sr, \id)$ and $(\sigma, \id)$ are given by construction. 
Consider the map $(tr,\id)$. 
We have that $(sr,\id)(\mathsf{Inv}^\sigma)=(sr,\id) \circ (tr, \id)(\mathsf{Inv}^\sigma)$.
But we know $(sr, \id)(\mathsf{Inv}^\sigma) \subseteq \mathsf{Inv}^\sigma$ and for any $v \notin \mathsf{Inv}^\sigma$, $(sr,\id)(v) \notin \mathsf{Inv}^\sigma$. Thus $(tr, \id)(\mathsf{Inv}^\sigma) \subseteq{\mathsf{Inv}^\sigma}$.

For (2), we have four possibilities: $(\id, \sigma)$, $(sr,\sigma)$, $(tr, \sigma)$, and $(\sigma, \sigma)$. 
By construction, we know that $(\id, \sigma)(\mathsf{Inv}^\sigma) \subseteq \mathsf{Inv}^\sigma$.
Then from (1), the remaining three can all be written as composites of maps $f \circ g$ such that $f(\mathsf{Inv}^\sigma) \subseteq \mathsf{Inv}^\sigma$ and $g(\mathsf{Inv}^\sigma) \subseteq \mathsf{Inv}^\sigma$. 

For (3), there are three possibilities to consider: $(sr, \id)$, $(tr,\id)$, and $(\sigma, \id)$.
$(sr,\id)$ is given by construction. Now, let $v \notin \mathsf{Inv}^\sigma$. 
We have that $(sr,\id)(v) = (sr, \id) \circ (tr,\id)(v)$. 
Thus $(tr, \id)(v) \notin \mathsf{Inv}^\sigma$, else we would have $(sr,\id)(v) \in \mathsf{Inv}^\sigma$, a contradiction. 
We also have that $v = (\sigma, \id) \circ (\sigma, \id)(v)$. So if $(\sigma, \id)(v) \in \mathsf{Inv}^\sigma$, we would get that $(\sigma, \id)(\mathsf{Inv}^\sigma) \not\subseteq \mathsf{Inv}^\sigma$, which is another contradiction. 

Finally, for (4), there are four more cases to consider: $(\id, \sigma)$, $(sr, \sigma)$, $(tr,\sigma)$, and $(\sigma, \sigma)$. 
But $(\id, \sigma)$ can be proven using an identical argument to $(\sigma, \id)$ in (3). 
Then from (3), the remaining three can be written as compositions of maps $f \circ g$, such that for any $v \notin \mathsf{Inv}^\sigma$ $f(v) \notin \mathsf{Inv}^\sigma$ and $g(v) \notin \mathsf{Inv}^\sigma$.
This completes the proof of the claim. 

Note that in a colimit diagram, every graph will have one of these sets of vertices $\mathsf{Inv}^\sigma$.
What we have just shown is that for a given graph, none of the vertices in $\mathsf{Inv}^\sigma$ get related to any other vertices by any maps of the form $(-,\id)$ or $(-, \sigma)$, except possibly vertices in $\mathsf{Inv}^\sigma$ for another graph. 
We use this fact to arrive at a contradiction when considering $I_0 \otimes I_3$.

Consider two vertices $v$ and $w$ connected by an edge in $I_3$. 
We know there exists a graph in the colimit diagram corresponding to the edge from $v$ to $w$.
Call this graph $\mathsf{Idty}_{vw}$.
For a graph of this form, we refer to the set $\mathsf{Inv}^\sigma \subset (\mathsf{Idty}_{vw})_V$ by $\mathsf{Inv}^\sigma_{vw}$.
There is also a graph $\mathsf{Idty}_{wv}$ corresponding to the edge going in the opposite direction, which is analogous to $\mathsf{Flip}$ in the case of $I_0 \otimes I_1$.
There also exist two more objects, one corresponding to the loops on each vertex $v$ and $w$.
Call these objects $\mathsf{Const}_v$ and $\mathsf{Const}_w$, respectively. 

We know that both $\mathsf{Idty}_{vw}$ and $\mathsf{Idty}_{wv}$ both a set of unlabelled vertices $\mathsf{Inv}^\sigma_{vw}$ and $\mathsf{Inv}^\sigma_{wv}$, respectively, which we constructed above. 
Our goal is to show that no vertex in $\mathsf{Inv}^\sigma_{vw} \cup \mathsf{Inv}^\sigma_{wv}$ belongs to an equivalence class containing any vertex not in $\mathsf{Inv}^\sigma_{vw} \cup \mathsf{Inv}^\sigma_{wv}$.
We showed above that any maps from either $\mathsf{Idty}_{vw}$, or $\mathsf{Idty}_{wv}$, to themselves, or between the two, do not relate the vertices in $\mathsf{Inv}^\sigma_{vw} \cup \mathsf{Inv}^\sigma_{wv}$ to any vertices not in $\mathsf{Inv}^\sigma_{vw} \cup \mathsf{Inv}^\sigma_{wv}$. 
This is because maps between these graphs, or to themselves, are of the form $(-,\id)$ and $(-,\sigma)$, as was the case in \cref{comma cat ex}.

There also do not exist any maps from either of these graphs to any other ones, and the only other maps to these graphs are from $\mathsf{Const}_v$ and $\mathsf{Const}_w$. 
These maps are of the form $(\id, sr)$, $(\id,tr)$, or $(ar,br)$ for $a,b \in \{s,t\}$. 
We show that these maps must not send any vertex in $\mathsf{Const}_v$ and $\mathsf{Const}_w$ to any vertex $a \in \mathsf{Inv}^\sigma_{vw}$, and the case of $\mathsf{Inv}^\sigma_{wv}$ is identical. 
The maps $(ar,br)$ are obvious since they're constant to the labelled vertices.
For the other maps, we have $(sr, \id) \circ (\id,sr) = (sr,sr)$ which must map all vertices to the vertex $(s,s)$. 
Then if $(\id,sr)(v) \in \mathsf{Inv}^\sigma_{vw}$, we'd get $(sr,sr)(v) \in \mathsf{Inv}^\sigma_{vw}$ since $(sr,\id)(\mathsf{Inv}^\sigma_{vw}) \subseteq \mathsf{Inv}^\sigma_{wv}$. This is a contradiction, so this cannot be the case.
The same can be said for $(tr, \id)$.

Thus, for a pair of distinct vertices $v,w$ with an edge between them, no vertex in $\mathsf{Inv}^\sigma_{vw} \cup \mathsf{Inv}^\sigma_{wv}$ belongs to an equivalence class containing any vertex not in $\mathsf{Inv}^\sigma_{vw} \cup \mathsf{Inv}^\sigma_{wv}$.

Now, for every edge $e$ from vertices $v$ to $w$ in $I_3$, choose a vertex $a_e$ in $\mathsf{Inv}^\sigma_{vw}$, and let $[a_e]$ be the equivalence class in the colimit that $a_e$ belongs to. 
Then, we can choose edges $e_1, e_2$ and $e_3$, such that no pair of these edges is between the same two vertices. We then have three distinct equivalence classes $[a_{e_1}]$, $[a_{e_2}]$, and $[a_{e_3}]$, such that there is not an edge between any two of them.
This is because if vertices $a$ and $b$ are in separate equivalence classes, they must be in separate graphs.
Then, in the product $I_0 \otimes I_3$, there exists a set of three vertices, none of which are connected by edges.
Since no such set exists in $I_3$, we have that $I_0 \otimes I_3 \neq I_3$, and thus $\otimes$ is not monoidal.
\end{proof}

\begin{corollary} \label{cor:exactly 4}
    The graph $F_\otimes(E,E)$ must be a graph of four vertices, such that each vertex has exactly one label. 
\end{corollary}

\begin{proof}
    By \cref{at most 4} we know that $F_\otimes(E,E)$ cannot have more than four vertices. 
    We also know by \cref{no double labels} that $F_\otimes(E,E)$ cannot have any double labels.
    Thus, $F_\otimes(E,E)$ must be a graph with exactly four vertices, as if it had fewer we would have at least one double label. 
    We also have that each vertex in $F_\otimes(E,E)$ must then have exactly one label, since if there was an unlabelled vertex, we would have to again have at least one double-labelled vertex.
\end{proof}

We are now ready to prove the main result of this paper:

\begin{theorem} \label{TH:monoidal structures}
Let $F_\otimes$ and $\otimes$ be as in \cref{assumption4}. Then either $\otimes = \boxtimes$ or $\otimes = \square$. 
\end{theorem}

\begin{proof}
We know by \cref{cor:exactly 4} that $F_\otimes(E,E)$ is a graph of four vertices all with distinct labels.
Note that since there are no vertices in $F_\otimes(E,E)$ with double labels, then there also cannot be vertices in $F_\otimes(E,V)$ or $F_\otimes(V,E)$ with double labels.
If this were the case, the maps from $F_\otimes(E,V)$ or $F_\otimes(V,E)$ would result in double labelled vertices in $F_\otimes(E,E)$. 
Thus, in both $F_\otimes(E,V)$ and $F_\otimes(V,E)$, one vertex must be labelled $s$ and the other $t$.

We now show that $F_\otimes(E,E)$ must be either $C_4$ or $K_4$. 
As always, label the four vertices $(s,s),(s,t),(t,s)$ and $(t,t)$.
We first show that we must have at least the condition that if $(a,a')$ and $(b,b')$ are such that $a=b$ or $a'=b'$ then $(a,a') \sim (b,b')$. 
First consider $(a,s)$ and $(a,t)$ (the other case is analogous). 
Suppose there is no edge between $(a,s)$ and $(a,t)$. 
We show that $F$ is not a functor. 
Consider the morphism $(a, \id) \colon  (V,E) \to (E,E)$. 
Then the image of this morphism must send the $s$ vertex of $I_1$ to $(a,s)$ and the $t$ vertex to $(a,t)$. 
But this is not a graph map since $s$ and $t$ have an edge between them in $I_1$, but $(a,s)$ and $(a,t)$ don't in $F_\otimes(E,E)$.
The case of $(s,a)$ and $(t,a)$ is identical.
Thus $F_\otimes$ is not a functor. 
So, if $F_\otimes$ is a functor, we must have that $(a,a') \sim (b,b')$ whenever $a=a'$ or $b=b'$.
Thus $C_4$ must be a subgraph of $F_\otimes(E,E)$.

Now suppose $F_\otimes(E,E)$ has one of the diagonal edges. Assume that $(s,s) \sim (t,t)$. 
We know that $(\sigma, \id)$ must be a graph map. 
But this map sends $(s,s)$ to $(t,s)$ and $(t,t)$ to $(s,t)$. 
Since it's a graph map we get that $(s,t) \sim (t,s)$. 
Similarly if $(s,t) \sim (t,s)$ we must have $(s,s) \sim (t,t)$. 
So if $F_\otimes(E,E)$ is not $C_4$, then it must be $K_4$.
Note that if $F_\otimes(E,E)$ is $K_4$ or $C_4$ with no unlabelled vertices, then the action of $F_\otimes$ on any morphism is fully determined by its action on $(s,s), (s,t),(t,s)$ and $(t,t)$ which are already defined.

Finally, by \cref{Prop: monoidal as kan extensions}, the first case corresponds to the functor resulting in the product $\square$ and the second case corresponds to the functor resulting in the product $\boxtimes$.
Thus, if $F_\otimes$ is a functor such that $\otimes$ is a closed symmetric monoidal structure, then $\otimes$ must be either $\boxtimes$ or $\square$. 
\end{proof}
\begin{theorem} \label{main th}
    There are only two closed symmetric monoidal products on the category $\Graph$: the box and categorical products.
\end{theorem}
\begin{proof}
By \cref{TH:monoidal structures}, we know that if a closed symmetric monoidal product $\otimes$ arises as a Kan extension $\otimes \iso \lan_{\yo \times \yo} F_\otimes$, then either $\otimes \iso \square$ or $\otimes \iso \boxtimes$. But, by \cref{Prop: monoidal as kan extensions}, every closed symmetric monoidal product must arise in this manner.
\end{proof}
 \bibliographystyle{amsalphaurlmod}
 \bibliography{all-refs.bib}

 \appendix
 \renewcommand{\thesection}{\Alph{section}}

\end{document}